\documentclass[11pt,a4paper]{amsart}

\usepackage{amssymb,enumitem,amsmath,amsthm,bbold,wrapfig}
\usepackage[all]{xy}
\usepackage{color,pxfonts}
\definecolor{Chocolat}{rgb}{0.36, 0.2, 0.09}
\definecolor{BleuTresFonce}{rgb}{0.215, 0.215, 0.36}
\definecolor{EgyptianBlue}{rgb}{0.06, 0.2, 0.65}

\usepackage{graphicx,scalerel}

\usepackage[colorlinks,final,hyperindex]{hyperref}
\hypersetup{citecolor=BleuTresFonce, urlcolor=EgyptianBlue, linkcolor=Chocolat}

\usepackage{fourier}

\catcode`\@=11

\catcode`\@=13

\newtheorem{theorem}{Theorem}
\newtheorem{lemma}{Lemma}

\newtheorem{proposition}{Proposition}
\newtheorem{conjecture}{Conjecture}
\newtheorem{question}{Question}

\theoremstyle{definition}

\DeclareMathAlphabet{\pazocal}{OMS}{zplm}{m}{n}

\DeclareMathAlphabet{\mathbbold}{U}{bbold}{m}{n}

\def\kk{\mathbbold{k}}

\begin{document}

\title[Nonassociative weak nilpotence]{Nilpotence, weak nilpotence, and the nil property in the nonassociative world: computations and conjectures}
\author{Vladimir Dotsenko}

\address{ 
Institut de Recherche Math\'ematique Avanc\'ee, UMR 7501, Universit\'e de Strasbourg et CNRS, 7 rue Ren\'e-Descartes, 67000 Strasbourg CEDEX, France}

\email{vdotsenko@unistra.fr}

\date{}

\begin{abstract}
We present some results, both rigorously mathematical and computational, showing unexpected relations between different identities expressing nilpotence in nonassociative algebras, and formulate a number of conjectural generalizations and related questions.
\end{abstract}

\maketitle

\section{Introduction}

Suppose that $A$ is a vector space equipped with a bilinear \emph{commutative} nonassociative binary operation $a_1,a_2\mapsto a_1a_2$.
Let us define an $n$-linear operation $t_n(a_1,\ldots,a_n)$ as the sum of all distinct multilinear terms of degree $n$ that one can define using this operation. For instance,
 \[
t_3(a_1,a_2,a_3)=(a_1a_2)a_3+(a_1a_3)a_2+a_1(a_2a_3),
 \]
and for the given $n$, the operation $t_n$ is a sum of $(2n-3)!!=1\cdot 3\cdots (2n-3)$ terms. If $t_d$ vanishes identically on $A$ for $d\ge n$, we shall say that the commutative algebra $A$ is weakly nil of index $n$. More generally, if $A$ has a totally commutative $k$-ary operation $a_1,\ldots,a_k\mapsto \phi(a_1 a_2 \ldots a_k)$, we define the operation $t_n^{(k)}(a_1,\ldots,a_n)$ to be the sum of all distinct multilinear terms of degree $n$ that one can define using this operation (here necessarily $n\equiv 1\pmod{k-1}$); vanishing of all the operations $t_d^{(k)}$ for $d\ge n$ defines weakly nil algebras of index $n$. For a ternary map, these conditions appear in Yagzhev's reformulation of the Jacobian Conjecture. Precise details can be found in \cite{Yag} (in Russian, not available online) and in two recent papers \cite{MR3229356,MR4038058}, but one of the key ideas is simple enough to be explained here.  Roughly speaking, if we consider a polynomial map $A\to A$ sending $x$ to $x-\phi(x,x,x)$, then existence of the polynomial inverse of this map is equivalent to the condition of weak nil property of some finite index to hold, and the constant determinant condition is equivalent to the so called Engel property of the operation $\phi(-,-,-)$, casting the Jacobian Conjecture for cubic maps (equivalent to the conjecture in its full glory, see \cite{MR663785,MR592226}) in the language of identities in ternary algebras. 

The main goal of this note is to suggest that there are natural noncommutative analogues of the elements $t_d^{(k)}$ which seem to have very subtle algebraic properties (here, once again, $n\equiv 1\pmod{k-1}$). Namely, if $A$ is a vector space equipped with a bilinear \emph{noncommutative} nonassociative binary operation 
 \[
a_1,a_2\mapsto a_1a_2,
 \]
we define an $n$-linear operation $w_n(a_1,\ldots,a_n)$ as the sum of all distinct multilinear terms of degree $n$ that one can define using this operation if putting the arguments $a_1,\ldots,a_n$ in the natural order. For instance,
 \[
w_3(a_1,a_2,a_3)=(a_1a_2)a_3+a_1(a_2a_3),
 \]
and for the given $n$, the operation $w_n$ is a sum of $\frac{1}{n}\binom{2n-2}{n-1}$ terms. More generally, if $A$ has a $k$-ary operation $a_1,\ldots,a_k\mapsto (a_1 a_2 \ldots a_k)$, we define the operation $w_n^{(k)}(a_1,\ldots,a_n)$ to be the sum of all distinct multilinear terms of degree $n$ that one can define using this operation if putting the arguments $a_1,\ldots,a_n$ in the natural order. The elements $w_{2k-1}^{(k)}$ originally emerged in questions of homotopical algebra and were used to construct nontrivial cycles in bar complexes of some operads \cite[Th.~9]{MR4106894}. 
In fact, each individual identity $w_d^{(k)}=0$ seems to be very powerful: for given $d$ and $k$, it seems to imply some unexpected nilpotence conditions, some of which are proved in this note, and some other conjectured. We show that the identity $w_d^{(k)}=0$ for a $k$-ary operation implies the identity $t_d^{(k)}=0$ for its total symmetrization, suggesting that the noncommutative story may help to shed some light on the commutative one. 

Our intuition and methods come from the operad theory, and some of the proofs heavily rely on experimental data obtained using the \texttt{Haskell} implementation \cite{OpGb} of Gr\"obner bases for operads \cite{MR3642294,MR2667136} (throughout this paper, we think of Gr\"obner bases as of  ``good'' systems of rewriting rules, since that is how they will be used). Since some of the conjectures below can probably benefit from different viewpoints, such as the structure theory of algebras of various types, the paper is written so that all statements and most conjectures can be comprehended by general mathematical audience not familiar with operads, and only certain proofs require fluency in some aspects of the operad theory. For simplicity, the ground field $\kk$ is assumed to be of zero characteristic; we do not have enough computational evidence to make characteristic-free claims, and we suspect that many of our conjectures are false in positive characteristic. There are three parts to this story: the case of a binary operation, where a lot of computational data is available, the case of a ternary operation where some computational data is available, and the extrinsic motivation alluded to in the beginning is present, and the case of a $k$-ary operation, where only some basic theoretical results are available. These three parts are presented in three separate sections of this paper.

\subsection*{Acknowledgements. } I am grateful to Murray Bremner who provided an independent verification of the first claim of Theorem \ref{th:BNil} using \texttt{Maple}, to Ualbai Umirbaev for encouraging him to write down these results, to Dmitry Piontkovski and Pasha Zusmanovich for useful discussions, and to Farukh Mashurov for editorial comments. The author's research was supported by Institut Universitaire de France, by the University of Strasbourg Institute for Advanced Study (USIAS) through the Fellowship USIAS-2021-061 within the French national program ``Investment for the future'' (IdEx-Unistra), and by the French national research agency project ANR-20-CE40-0016. This paper was written during the author's stay at Max Planck Institute for Mathematics in Bonn, and he wishes to express his gratitude to that institution for the financial support and excellent working conditions. 

\section{Binary nonassociative product}

In this section, $A$ is an algebra with a (not necessarily associative) product $A\otimes A\to A$. 

\subsection{Noncommutative case}

As explained in Introduction, we shall consider, for each $n$, the multilinear operation $w_n(a_1,\ldots,a_n)$ defined as the sum of all distinct multilinear terms of degree $n$ that one can define using the product in $A$ if putting the arguments $a_1,\ldots,a_n$ in the natural order. It can be equivalently defined inductively as follows: we set $w_1(a_1)=a_1$ and
 \[
w_n(a_1,\ldots,a_n)=\sum_{i=1}^{n-1}w_i(a_1,\ldots,a_i)w_{n-i}(a_{i+1},\ldots,a_n).
 \]
We say that an algebra $A$ is \emph{weakly nilpotent of index $n$} if the identity $w_d=0$ holds in $A$ for all $d\ge n$. 

The first natural question to ask here is to determine a minimal finite system of identities defining weakly nilpotent algebras of index $n$. 

\begin{proposition}\leavevmode
\begin{itemize}
\item For each $n$, the identities $w_n=w_{n+1}=\cdots=w_{2n-2}=0$ imply weak nilpotence of index $n$.
\item The identity $w_3=0$ implies weak nilpotence of index $3$.
\item The identity $w_4=0$ does not imply $w_5=0$, and thus does not imply weak nilpotence of index $4$. 
\item The identities $w_4=w_5=0$ imply weak nilpotence $4$. 
\end{itemize}
\end{proposition}

\begin{proof}
The first statement is immediate from the inductive definition of our operations: for $k\ge 2n-1$, the operation $w_k$ is equal to the sum of terms in which each term contains $w_d$ for $n\le d<k$, so one may argue by induction. The second statement follows from the well known result of nilpotence of anti-associative algebras \cite{MR1489904}. The last two statements rely on a computer calculation using the \texttt{Haskell} program \cite{OpGb} for computing Gr\"obner bases for operads. Namely, it turns out that the reduced Gr\"obner bases of the nonsymmetric operads encoding the identity $w_4=0$ and the identities $w_4=w_5=0$ are different, showing that the two systems of identities are not equivalent. The former one is given in Theorem \ref{th:BNil} below, and the second one will be used here to show that the identities $w_4=w_5=0$ imply the only ``missing'' identity $w_6=0$. 

\begin{lemma}
For the graded path-lexicographic order of the free operad, the reduced Gr\"obner basis for the nonsymmetric operad encoding the identities $w_4=w_5=0$ gives the following set of rewriting rules: 
\begin{itemize}
\item one element of arity $4$ 
 \[
(((a_1 a_2) a_3) a_4) \mapsto  - ((a_1 (a_2 a_3)) a_4)  -  ((a_1 a_2) (a_3 a_4))  -  (a_1 ((a_2 a_3) a_4))  -  (a_1 (a_2 (a_3 a_4))),
 \]
\item two elements of arity $5$
\begin{gather}
((a_1 a_2) ((a_3 a_4) a_5))  \mapsto  (a_1 ((a_2 a_3) (a_4 a_5)))  +  (a_1 (a_2 (a_3 (a_4 a_5)))),\\
((a_1 (a_2 (a_3 a_4))) a_5) \mapsto - (a_1 ((a_2 (a_3 a_4)) a_5))  -  (a_1 (a_2 ((a_3 a_4) a_5)))  -  (a_1 (a_2 (a_3 (a_4 a_5)))),
\end{gather}
\item and four elements of arity $6$
\begin{gather}
((a_1 a_2) (a_3 ((a_4 a_5) a_6))) \mapsto  - ((a_1 a_2) (a_3 (a_4 (a_5 a_6))))  -  (a_1 (a_2 (a_3 ((a_4 a_5) a_6)))),\\
(a_1 ((a_2 a_3) (a_4 (a_5 a_6)))) \mapsto  - (a_1 (a_2 ((a_3 a_4) (a_5 a_6)))),\\
(a_1 ((a_2 (a_3 a_4)) (a_5 a_6))) \mapsto  - (a_1 (a_2 ((a_3 a_4) (a_5 a_6)))),\\
(a_1 (a_2 (a_3 (a_4 (a_5 a_6))))) \mapsto 0 .
\end{gather}
\end{itemize}
\end{lemma}

In particular, for in this case the algebra generated by left multiplications $l_a\colon x\mapsto ax$ is nilpotent of index $5$, which, once compared with the Gr\"obner basis discussed in the proof of Theorem \ref{th:BNil}, shows that the identity $w_5=0$ is not a consequence of the identity $w_4=0$ . 

Having a Gr\"obner basis for a system of identities, there is a very easy algorithm that allows to check if another identity follows from them. First, one may assume the identity homogeneous and multilinear (since we work over a field of characteristic zero), and also one may assume that all variables appear in it in the standard order, since we start from identities of that sort (in other words, we work with a nonsymmetric operad). That algorithm proceeds as follows. It examines monomials that appear in our identity, and for each monomial that is ``divisible'' (in the sense of operads/verbal ideals) by a monomial on the left hand side of some rewriting rule, it reduces that monomial using the corresponding rewriting rule. After finitely many steps (bounded, for identities of arity $n$, by the total dimension of the space of identities of that arity), one shall obtain an identity that cannot be reduced further. The original identity is a consequence of the given ones if and only if the irreducible identity thus obtained is zero. 

In particular, using our Gr\"obner basis, it is easy to check that $w_6=0$ follows from these two identities, and so, according to the first statement, any algebra satisfying the identities $w_4=w_5=0$ is weakly nilpotent of index $4$.
\end{proof}

We conjecture that the same phenomenon holds for each $n$.

\begin{conjecture}
For each $n$, the identities $w_n=w_{n+1}=\cdots=w_{2n-3}=0$ imply the identity $w_{2n-2}=0$, and hence imply weak nilpotence of index $n$.
\end{conjecture}

Our next result shows that the identity $w_4=0$ implies some unexpected nilpotence for left multiplications.

\begin{theorem}\label{th:BNil}
Let $A$ be an algebra satisfying the identity $w_4=0$.
\begin{enumerate}
\item The associative algebras $L(A)$ generated by left multiplications $l_a\colon x\mapsto ax$ and $R(A)$ generated by right multiplications $r_a\colon x\mapsto xa$ are nilpotent; in both cases, the nilpotence is of index $7$ (and not less, in general).
\item In general, the multiplicative universal enveloping algebra $U(A)$ generated by both left and right multiplications is not nilpotent. 
\end{enumerate}
\end{theorem}

\begin{proof}
Using the \texttt{Haskell} program \cite{OpGb}\footnote{Some of the results of the computation were verified by Murray Bremner using \texttt{Maple}.}, we, once again, found a finite Gr\"obner basis.

\begin{lemma}
For the graded path-lexicographic order of the free operad, the reduced Gr\"obner basis for the nonsymmetric operad encoding the identity $w_4=0$ gives the following set of rewriting rules: 
\begin{itemize}
\item one element of arity $4$
 \[
(((a_1 a_2) a_3) a_4) \mapsto  - ((a_1 (a_2 a_3)) a_4)  -  ((a_1 a_2) (a_3 a_4))  -  (a_1 ((a_2 a_3) a_4))  -  (a_1 (a_2 (a_3 a_4))),
 \]
\item one element of arity $5$ 
\begin{multline}
((a_1 (a_2 (a_3 a_4))) a_5) \mapsto ((a_1 a_2) ((a_3 a_4) a_5))  -  (a_1 ((a_2 (a_3 a_4)) a_5))  -  (a_1 ((a_2 a_3) (a_4 a_5)))\\  -  (a_1 (a_2 ((a_3 a_4) a_5)))  -  2 (a_1 (a_2 (a_3 (a_4 a_5)))),
\end{multline}
\item two elements of arity $6$
\begin{multline}
((a_1 a_2) ((a_3 a_4) (a_5 a_6))) \mapsto  - ((a_1 a_2) (a_3 ((a_4 a_5) a_6)))  -  ((a_1 a_2) (a_3 (a_4 (a_5 a_6)))) \\ -  (a_1 ((a_2 a_3) ((a_4 a_5) a_6)))  -  (a_1 (a_2 (a_3 ((a_4 a_5) a_6))))  +  (a_1 (a_2 (a_3 (a_4 (a_5 a_6))))),
\end{multline}
\begin{multline}
((a_1 (a_2 a_3)) ((a_4 a_5) a_6)) \mapsto ((a_1 a_2) ((a_3 (a_4 a_5)) a_6))  -  ((a_1 a_2) (a_3 ((a_4 a_5) a_6))) \\ -  ((a_1 a_2) (a_3 (a_4 (a_5 a_6))))  -  (a_1 ((a_2 a_3) ((a_4 a_5) a_6))) \\ -  2 (a_1 (a_2 (a_3 ((a_4 a_5) a_6))))  -  (a_1 (a_2 (a_3 (a_4 (a_5 a_6))))),
\end{multline}
\item eight elements of arity $7$
\begin{multline}
(a_1 ((a_2 a_3) ((a_4 (a_5 a_6)) a_7)))  \mapsto  (a_1 (a_2 ((a_3 a_4) ((a_5 a_6) a_7))))  -  (a_1 (a_2 (a_3 ((a_4 a_5) (a_6 a_7)))))\\  +  (a_1 (a_2 (a_3 (a_4 ((a_5 a_6) a_7)))))  -  (a_1 (a_2 (a_3 (a_4 (a_5 (a_6 a_7)))))),
\end{multline}
\begin{multline}
((a_1 a_2) (a_3 ((a_4 a_5) (a_6 a_7))))  \mapsto  (a_1 (a_2 ((a_3 a_4) ((a_5 a_6) a_7))))  -  (a_1 (a_2 (a_3 ((a_4 a_5) (a_6 a_7)))))\\  +  (a_1 (a_2 (a_3 (a_4 ((a_5 a_6) a_7)))))  - (a_1 (a_2 (a_3 (a_4 (a_5 (a_6 a_7)))))),
\end{multline}
\begin{multline}
((a_1 (a_2 a_3)) (a_4 ((a_5 a_6) a_7)))  \mapsto  - ((a_1 (a_2 a_3)) (a_4 (a_5 (a_6 a_7))))  +  (a_1 (a_2 ((a_3 a_4) (a_5 (a_6 a_7))))) \\ +  (a_1 (a_2 (a_3 ((a_4 a_5) (a_6 a_7)))))  +  (a_1 (a_2 (a_3 (a_4 ((a_5 a_6) a_7)))))  +  (a_1 (a_2 (a_3 (a_4 (a_5 (a_6 a_7)))))),
\end{multline}
\begin{multline}
(a_1 ((a_2 ((a_3 a_4) a_5)) (a_6 a_7)))  \mapsto  (a_1 (a_2 ((a_3 (a_4 a_5)) (a_6 a_7))))  +  (a_1 (a_2 ((a_3 a_4) (a_5 (a_6 a_7)))))\\  +  2 (a_1 (a_2 (a_3 ((a_4 a_5) (a_6 a_7)))))  +  (a_1 (a_2 (a_3 (a_4 ((a_5 a_6) a_7)))))  +  2 (a_1 (a_2 (a_3 (a_4 (a_5 (a_6 a_7)))))),
\end{multline}
\begin{multline}
((a_1 a_2) (a_3 (a_4 (a_5 (a_6 a_7)))))  \mapsto  - (a_1 ((a_2 a_3) (a_4 (a_5 (a_6 a_7)))))  -  (a_1 (a_2 ((a_3 a_4) (a_5 (a_6 a_7)))))\\  -  (a_1 (a_2 (a_3 ((a_4 a_5) (a_6 a_7)))))  -  (a_1 (a_2 (a_3 (a_4 ((a_5 a_6) a_7)))))  -  2 (a_1 (a_2 (a_3 (a_4 (a_5 (a_6 a_7)))))),
\end{multline}
\begin{multline}
((a_1 a_2) ((a_3 ((a_4 a_5) a_6)) a_7))  \mapsto  (a_1 (a_2 ((a_3 a_4) ((a_5 a_6) a_7))))  -  (a_1 (a_2 (a_3 ((a_4 (a_5 a_6)) a_7))))\\  -  (a_1 (a_2 (a_3 ((a_4 a_5) (a_6 a_7)))))  -  (a_1 (a_2 (a_3 (a_4 ((a_5 a_6) a_7)))))  -  2 (a_1 (a_2 (a_3 (a_4 (a_5 (a_6 a_7)))))),
\end{multline}
\begin{multline}
(a_1 ((a_2 a_3) (a_4 ((a_5 a_6) a_7))))  \mapsto  - (a_1 ((a_2 a_3) (a_4 (a_5 (a_6 a_7)))))  +  (a_1 (a_2 ((a_3 a_4) ((a_5 a_6) a_7))))\\  -  (a_1 (a_2 ((a_3 a_4) ((a_5 a_6) a_7))))  -  2 (a_1 (a_2 (a_3 ((a_4 a_5) (a_6 a_7)))))  -  (a_1 (a_2 (a_3 (a_4 ((a_5 a_6) a_7)))))\\  -  3 (a_1 (a_2 (a_3 (a_4 (a_5 (a_6 a_7)))))),
\end{multline}
\begin{multline}
((a_1 a_2) (a_3 ((a_4 (a_5 a_6)) a_7)))  \mapsto  - ((a_1 a_2) (a_3 (a_4 ((a_5 a_6) a_7))))  +  (a_1 ((a_2 a_3) (a_4 (a_5 (a_6 a_7)))))\\  -  (a_1 (a_2 ((a_3 a_4) ((a_5 a_6) a_7))))  +  (a_1 (a_2 ((a_3 a_4) (a_5 (a_6 a_7)))))  -  (a_1 (a_2 (a_3 ((a_4 (a_5 a_6)) a_7)))) \\ +  2 (a_1 (a_2 (a_3 ((a_4 a_5) (a_6 a_7)))))  +  3 (a_1 (a_2 (a_3 (a_4 (a_5 (a_6 a_7)))))),
\end{multline}
\item and eleven elements of arity $8$
\begin{gather}
  (a_1 (a_2 (a_3 (a_4 (a_5 ((a_6 a_7) a_8)))))) \mapsto0,\\
  (a_1 (a_2 (a_3 (a_4 ((a_5 (a_6 a_7)) a_8))))) \mapsto0,\\
  (a_1 (a_2 (a_3 ((a_4 a_5) ((a_6 a_7) a_8))))) \mapsto0,\\
  (a_1 (a_2 ((a_3 a_4) (a_5 (a_6 (a_7 a_8)))))) \mapsto0,\\
  (a_1 (a_2 ((a_3 (a_4 a_5)) (a_6 (a_7 a_8))))) \mapsto0,\\
  (a_1 ((a_2 a_3) (a_4 (a_5 ((a_6 a_7) a_8))))) \mapsto0,\\
  (a_1 ((a_2 (a_3 a_4)) (a_5 (a_6 (a_7 a_8))))) \mapsto0,\\
  (a_1 (a_2 (a_3 ((a_4 (a_5 a_6)) (a_7 a_8))))) \mapsto0,\\
  (a_1 (a_2 (a_3 ((a_4 a_5) (a_6 (a_7 a_8)))))) \mapsto0,\\
  (a_1 (a_2 (a_3 (a_4 ((a_5 a_6) (a_7 a_8)))))) \mapsto0,\\
  (a_1 (a_2 (a_3 (a_4 (a_5 (a_6 (a_7 a_8))))))) \mapsto0.\label{eq:lnil}
\end{gather}
\end{itemize}
\end{lemma}

Using this Gr\"obner basis, we may now prove both claims. 

\smallskip

1. The rewriting rule \eqref{eq:lnil} reads $l_{a_1}l_{a_2}l_{a_3}l_{a_4}l_{a_5}l_{a_6}l_{a_7}(a_8)=0$, so the algebra $L(A)$ is nilpotent of index $7$. Of course, an algebra is weakly nilpotent of some index $d$ if and only if the opposite algebra is weakly nilpotent of index $d$, so the same holds for right multiplications. At the same time, the element 
 \[
(a_1 (a_2 (a_3 (a_4 (a_5 (a_6 a_7))))))
 \] 
is not divisible by any left hand side of rewriting rules, so it does not vanish modulo $w_4=0$.  

\smallskip

2. Let us define the ``zig-zag'' monomials $m_i$ recursively by putting 
\begin{gather}
m_1=a_1,\\
m_{2k}(a_1,\ldots,a_{2k})=m_{2k-1}(a_1,\ldots,a_{2k-1})a_{2k},\\
m_{2k+1}(a_1,\ldots,a_{2k+1})=a_1m_{2k}(a_2,\ldots,a_{2k}).
\end{gather}
By a direct inspection, these monomials are not divisible by any left hand side of rewriting rules, so they do not vanish modulo $w_4=0$, and thus the multiplicative universal enveloping algebra $U(A)$ is, in general, not nilpotent; moreover, this shows that its subalgebra generated by elements $l_{a}r_{b}$ is not nilpotent.  
\end{proof}

The following conjecture is ``reasonable'' in that it was checked on a computer once (though the computation took a few days, and no independent verification followed). 

\begin{conjecture}
For every algebra $A$ satisfying the identity $w_5=0$, the associative algebras $L(A)$ and $R(A)$ are nilpotent; in both cases, the nilpotence is of index $15$ (and not less, in general).
\end{conjecture}

Guided by some general heuristics on operads, we also make the following conjecture.

\begin{conjecture}
For each $d\in\mathbb{N}$, the reduced Gr\"obner basis for the nonsymmetric operad encoding the identity $w_d=0$ is finite.
\end{conjecture}

It is also natural to raise the following question.

\begin{question}
Let $d\in\mathbb{N}$. Is it true that for every algebra $A$ satisfying the identity $w_d=0$, the associative algebras $L(A)$ and $R(A)$ are nilpotent? If yes, is it true that the nilpotence is of index $2^{d-1}-1$ (and not less, in general)?
\end{question}

\subsection{Commutative case}

Throughout this section, we assume that the algebra $A$ is commutative. As explained in Introduction, we shall consider, for each $n$, the multilinear operation $t_n(a_1,\ldots,a_n)$ defined as the sum of all distinct multilinear terms of degree $n$ that one can define using the product in $A$ (sum of all \emph{shuffle} binary trees, see \cite{MR3642294}). We say that an algebra $A$ is \emph{weakly nil of index $n$} if the identity $t_d=0$ holds in $A$ for all $d\ge n$. 

Let us start with proving the result alluded to in the introduction that relates the weak nil property to the weak nilpotence.

\begin{proposition}
Suppose that $A$ is a nonassociative algebra that satisfies the identity $w_d=0$. Then the symmetrized product $a_1\circ a_2=a_1a_2+a_2a_1$ satisfies the identity $t_d=0$. 
\end{proposition}

\begin{proof}
This is a direct algebraic manipulation: one should add all permutations of the identity $w_d=0$ and collect terms together; the result is a scalar multiple of the identity $t_d=0$ for the symmetrized product. 
\end{proof}

Thus, we obtain a functor $A\mapsto A^{(+)}$ that relates the two types of algebras. For the identity $w_3=0$, this is a well known result that seem to go back to \cite{MR1489904}: in an algebra satisfying the anti-associativity identity $(a_1a_2)a_3+a_1(a_2a_3)=0$, the identity $t_3=0$ (known as the mock-Lie identity \cite{MR1358617}, or the Lie-Jordan identity \cite{MR3247244}) holds for the symmetrized product. 
For $d=3$, our functor behaves rather pathologically: anti-associative algebras are nilpotent of index $4$, and mock-Lie algebras are not, so we cannot hope to represent the identity $t_3=0$ faithfully. The reader is invited to consult \cite{MR3598575} for a further discussion of representability. Nevertheless, we feel that this relationship between the two types of nilpotence might be useful. Of course, it would be important to understand the behaviour of our functor for $d>3$.

\begin{question}
Let $A$ be the free algebra for the identity $w_d=0$ generated by $x_1,\ldots,x_p$ for some $p\in\mathbb{N}$. For which $d>3$ is the subalgebra of $A$ generated by $x_1,\ldots,x_p$ under the symmetrized product $a_1\circ a_2=a_1a_2+a_2a_1$ free for the identity $t_d=0$?
\end{question}

As above, one of the first natural questions to ask here is to determine a minimal finite system of identities defining weakly nil algebras of index $n$. 

\begin{proposition}\leavevmode
\begin{itemize}
\item For each $n$, the identities $t_n=t_{n+1}=\cdots=t_{2n-2}=0$ imply the weak nil property of index $n$.
\item The identity $t_3=0$ implies the weak nil property of index $3$.
\end{itemize}
\end{proposition}

\begin{proof}
The first claim is analogous to the one in the noncommutative case, since the elements $t_n$ also admit an inductive description: we have $t_1(a_1)=a_1$ and
 \[
t_n(a_1,\ldots,a_n)=\sum_{i=1}^{n-1}\sum_{\substack{I\sqcup J=\{1,\ldots,n\},\\ |I|=i, |J|=n-i, 1\in I}}t_i(a_I)t_{n-i}(a_J).
 \]
The second claim only requires to check that $t_4=0$ follows from $t_3=0$, which is checked by a direct computation. 
\end{proof}

The following conjecture is inspired by what we know in the noncommutative case. 
\begin{conjecture}\leavevmode
\begin{enumerate}
\item The identity $t_4=0$ does not imply $t_5=0$, and thus does not imply the index $4$ weak nil property. 
\item The identities $t_4=t_5=0$ imply the weak nil property of index $4$. 
\item For each $n$, the identities $t_n=t_{n+1}=\cdots=t_{2n-3}=0$ imply the identity $t_{2n-2}=0$, and hence imply weak nilpotence of index $n$.
\end{enumerate}
\end{conjecture}

It is known (see, for example, \cite{MR1278792}) that the identity $t_3=0$ imply the $3$-Engel identity $(((xy)y)y)=0$. Guided by the noncommutative case, we make the following conjecture.

\begin{conjecture}
The identity $t_4=0$ implies the $7$-Engel identity. 
\end{conjecture}

We remark that some other much more unexpected identities were verified in \cite{MR1278792} by computer-assisted methods; for example, it is shown that the identity $t_3=0$ implies the really surprising identity $((((a_1a_2)(a_3a_4))(a_5a_6))(a_7a_8))a_9=0$. This prompts the following question.

\begin{question}
Let $d\in\mathbb{N}$. Does the identity $t_d=0$ always imply that certain multilinear monomial vanishes?
\end{question}

Finally, the noncommutative case also makes us raise the following question.

\begin{question}
Let $d\in\mathbb{N}$. Is it true that every commutative algebra $A$ satisfying the identity $t_d=0$ satisfies the $m$-Engel identity for some $m$?
\end{question}

The converse of this question is stated as an open problem in \cite{MR3229356}; it would imply a strengthening of the (known case of) Jacobian Conjecture for quadratic maps.

To conclude this section, let us remark that it would also be interesting to apply our approach to the nilpotency theorems of \cite[Sec.~8]{MR4513786}.

\section{Ternary product}

In this section, $A$ is an algebra with a ternary product $A\otimes A\otimes A\to A$, $a_1,a_2,a_3\mapsto (a_1 a_2 a_3)$. 

\subsection{Noncommutative case}

As explained in Introduction, for each odd $n$, we may consider the operation $w_n^{(3)}(a_1,\ldots,a_n)$ as the sum of all distinct multilinear terms of degree $n$ that one can define using the ternary product in $A$ if putting the arguments $a_1,\ldots,a_n$ in the natural order. It can be equivalently defined inductively as follows:
\begin{gather}
w_1^{(3)}(a_1)=a_1,\\
w_n^{(3)}(a_1,\ldots,a_n)=\sum_{i+j+k=n}(w_i^{(3)}(a_1,\ldots,a_i) w_j(a_{i+1},\ldots,a_{i+j}) w_k^{(3)}(a_{i+j+1},\ldots,a_n)).
\end{gather}
We say that an algebra $A$ is \emph{weakly nilpotent of index $n$} if the identity $w_d^{(3)}(a_1,\ldots,a_d)$ holds in $A$ for all odd $d\ge n$. 

Our first question concerns a minimal set of identities ensuring weak nilpotence. As in the binary case, one can immediately see that identities $w_{2n-1}n=w_{2n+1}=\cdots=w_{6n-3}=0$ imply weak nilpotence of index $2n-1$. 

\begin{question}
For the given $n\in\mathbb{N}$, describe a minimal set of identities among $w_i^{(3)}=0$ that imply weak nilpotence of index $2n-1$.
\end{question}

The following result is an analogue of the ``unexpected nilpotence'' observed in the binary case.

\begin{theorem}\label{th:3Nil}
Let $A$ be a ternary algebra satisfying the identity $w_5^{(3)}=0$. 
\begin{enumerate}
\item The associative algebra $L(A)$ generated by left multiplications $l_{a,b}\colon x\mapsto (a b x)$ is nilpotent and the associative algebra $R(A)$ generated by right multiplications $r_{a,b}\colon x\mapsto (x a b)$ is nilpotent, in both cases the nilpotence is of index $4$ (and not less, in general).
\item In general, the multiplicative universal enveloping algebra $U(A)$ generated by both left and right multiplications is not nilpotent. 
\end{enumerate}
\end{theorem}

\begin{proof}
The Gr\"obner basis in this case is again finite, and leads to the following rewriting rules:
\begin{itemize}
\item one rewriting rule of arity $5$: $$((a_1 a_2 a_3) a_4 a_5) \mapsto  - (a_1 (a_2 a_3 a_4) a_5)  -  (a_1 a_2 (a_3 a_4 a_5)),$$
\item one rewriting rule of arity $7$ : $$(a_1 (a_2 a_3 (a_4 a_5 a_6)) a_7) \mapsto  - (a_1 a_2 (a_3 (a_4 a_5 a_6) a_7))  -  (a_1 a_2 (a_3 a_4 (a_5 a_6 a_7))),$$
\item and three rewriting rules of arity $9$:
\begin{gather}
(a_1 a_2 (a_3 a_4 (a_5 a_6 (a_7 a_8 a_9))))\mapsto 0,\\
(a_1 a_2 (a_3 (a_4 a_5 a_6) (a_7 a_8 a_9)))\mapsto 0,\\
(a_1 (a_2 a_3 a_4) (a_5 (a_6 a_7 a_8) a_9))\mapsto  - (a_1 (a_2 a_3 a_4) (a_5 a_6 (a_7 a_8 a_9))).
\end{gather}
\end{itemize}
 
The claims follow in the same way as for the binary case.
\end{proof}

The following conjecture is a speculation based on the known result for the binary case. 

\begin{conjecture}
For every ternary algebra $A$ satisfying the identity $w_7^{(3)}=0$, the associative algebra $L(A)$ generated by left multiplications is nilpotent and the associative algebra $R(A)$ generated by right multiplications are nilpotent, in both cases the nilpotence is of index $13$ (and not less, in general).
\end{conjecture}

We also believe that the following result is true.

\begin{conjecture}
For each $d\in\mathbb{N}$, the reduced Gr\"obner basis for the nonsymmetric operad encoding the identity $w_{2d+1}^{(3)}=0$ is finite.
\end{conjecture}

It is also natural to raise the following question.

\begin{question}
Let $d\in\mathbb{N}$. Is it true that for every algebra $A$ satisfying the identity $w_{2d+1}^{(3)}=0$, the associative algebras $L(A)$ and $R(A)$ are nilpotent? 
\end{question}

\subsection{Commutative case}

Throughout this section, we assume that the algebra $A$ is commutative, and we denote the ternary product by $[a_1 a_2 a_3]$. We shall consider, for each $n$, the multilinear operation $t_n^{(3)}(a_1,\ldots,a_n)$ defined as the sum of all distinct multilinear terms of degree $n$ that one can define using the product in $A$ (sum of all \emph{shuffle} ternary trees). We say that an algebra $A$ is \emph{weakly nil of index $n$} if the identity $t_d^{(3)}=0$ holds in $A$ for all $d\ge n$. 

As in the binary case, the following result is established by a direct algebraic manipulation. 

\begin{proposition}
Suppose that $A$ is a ternary algebra satisfying the identity $w_{2d+1}^{(3)}=0$. Then the symmetrized product 
 \[
[a_1 a_2 a_3]=(a_1 a_2 a_3) + (a_1 a_3 a_2) + (a_2 a_3 a_1) + (a_2 a_1 a_3) + (a_3 a_1 a_2) + (a_3 a_2 a_1)
 \]
satisfies the identity $t_{2d+1}^{(3)}=0$. 
\end{proposition}

This leads to the following natural question. 

\begin{question}
Let $A$ be the free algebra for the identity $w_{2d+1}^{(3)}=0$ generated by $x_1,\ldots,x_p$ for some $p\in\mathbb{N}$. For which $d\in\mathbb{N}$ and is the subalgebra of $A$ generated by $x_1,\ldots,x_p$ under the symmetrized product $[-,-,-]$ free for the identity $t_{2d+1}^{(3)}=0$?
\end{question}

To conclude this section, we illustrate another facet of the operadic viewpoint by explaining an interpretation of another way of dealing with totally commutative ternary algebras proposed by Yagzhev in \cite{Yag}. He remarked, that one can relate properties of the polynomial map $x\mapsto x-[x,x,x]$ to properties of the polynomial map $x\mapsto x-(x\cdot x)\cdot x$, where $(-\cdot-)$ is a bilinear operation satisfying the properties
\begin{gather}
x\cdot (y\cdot y)=-(y\cdot y)\cdot x,\label{eq:yag1}\\
(x\cdot y+y\cdot x)y=2(y\cdot y)\cdot x,\label{eq:yag2}
\end{gather} 
as well as an identity of the form $(\cdots((x\cdot(y\cdot y))\cdot(y\cdot y))\cdots)\cdot(y\cdot y)=0$. He explains the equivalence in terms of constructing one finite-dimensional algebra out of another one. However, this is a consequence of a formal algebraic manipulation on the level of operads, which suggests that, though psychologically easier, this change probably does not reduce the complexity of the problem.

\begin{proposition}
For a binary operation satisfying identities \eqref{eq:yag1}, \eqref{eq:yag2}, the operation $\phi(x,y,z)=x\cdot (y\cdot z+z\cdot y)-(y\cdot z+z\cdot y)\cdot x$ is totally commutative and does not satisfy any other identity. Moreover, an identity of the form 
an identity of the form $(\cdots((x\cdot(y\cdot y))\cdot(y\cdot y))\cdots)\cdot(y\cdot y)=0$ is precisely the Engel property 
 \[
\phi(\phi(\cdots\phi(\phi(x,y,y),y,y)\cdots),y,y)=0.
\] 
\end{proposition}

\begin{proof}
It is convenient to introduce the operations $a_1\bullet a_2 =a_1\cdot a_2+a_2\cdot a_1$, $\{a_1,a_2\}:=a_1\cdot a_2-a_2\cdot a_1$. The multilinearization of the first identity becomes $(a_1\bullet a_2)\bullet a_3=0$, and the multilinearization of the second identity becomes, modulo the first one, $$\{a_1\bullet a_2,a_3\}+\{a_1\bullet a_3,a_2\}=2\{a_2\bullet a_3,a_1\},$$ which is equivalent to the simpler $\{a_1\bullet a_2,a_3\}=\{a_1\bullet a_3,a_2\}$, clearly implying the total commutativity of the operation $\phi$. Moreover, once translated into the language of the shuffle operads, the corresponding relations are immediately checked to form a Gr\"obner basis, easily implying that the suboperad generated by the  operation $\phi$ is free. The statement about the Engel property is obvious.
\end{proof}

\section{Products of higher arities}

In this section, $A$ is a $k$-ary algebra with the operation $a_1,\ldots,a_k\mapsto (a_1 a_2 \ldots a_k)$. As explained in Introduction, we consider the operation $w_n^{(k)}(a_1,\ldots,a_n)$ as the sum of all distinct multilinear terms of degree $n$ that one can define using this operation if putting the arguments $a_1,\ldots,a_n$ in the natural order.

\begin{theorem}\label{th:KNil}
Let $A$ be an algebra satisfying the identity $w_{2k-1}^{(k)}=0$. The associative algebras $L(A)$ generated by left multiplications $l_{a_1,\ldots,a_{k-1}}\colon x\mapsto (a_1 \ldots a_{k-1} x)$ and $R(A)$ generated by right multiplications $r_{a_1,\ldots,a_{k-1}}\colon x\mapsto (x a_1 \ldots a_{k-1})$ are nilpotent of index $k+1$.
\end{theorem}

\begin{proof} 
This is the result of \cite[Th.~10]{MR4106894}.
\end{proof}

We have several conjectures concerning these identities. 

\begin{conjecture}\leavevmode
\begin{itemize}
\item The bound on the nilpotence index of $L(A)$ and $R(A)$ in Theorem \ref{th:KNil} is tight.
\item For all $k\in\mathbb{N}$ and $d\equiv1\pmod{k-1}$, the reduced Gr\"obner basis for the nonsymmetric operad encoding the identity $w_d^{(k)}=0$ is finite.
\end{itemize}
\end{conjecture}

It is also natural to raise the following question.

\begin{question}
Let $k\in\mathbb{N}$ and $d\equiv1\pmod{k-1}$. Is it true that for every algebra $A$ satisfying the identity $w_d^{(k)}=0$, the associative algebras $L(A)$ and $R(A)$ are nilpotent? 
\end{question}

If we work with a $k$-ary algebra that is totally commutative, with the product $[a_1\ldots a_k]$, we shall consider, for each $n$, the multilinear operation $t_n^{(k)}(a_1,\ldots,a_n)$ defined as the sum of all distinct multilinear terms of degree $n$ that one can define using the product in $A$ (sum of all \emph{shuffle} $k$-ary  trees). We say that an algebra $A$ is \emph{weakly nil of index $n$} if the identity $t_d^{(k)}=0$ holds in $A$ for all $d\ge n$. 

As in the binary case, the following result is established by a direct algebraic manipulation. 

\begin{proposition}
Suppose that $A$ is a $k$-ary algebra that satisfies the identity $w_d^{(k)}=0$. Then the symmetrized product 
 \[
[a_1 \ldots a_k]=\sum_{\sigma\in S_k}(a_{\sigma(1)} \ldots a_{\sigma(k)})
 \]
satisfies the identity $t_d^{(k)}=0$. 
\end{proposition}

This leads to the following natural question. 

\begin{question}
Let $A$ be the free algebra for the identity $w_d^{(k)}=0$ generated by $x_1,\ldots,x_p$ for some $p\in\mathbb{N}$. For which $k\in\mathbb{N}$ and $d\equiv1\pmod{k-1}$ is the subalgebra of $A$ generated by $x_1,\ldots,x_p$ under the symmetrized product $[-,\ldots,-]$ free for the identity $t_d^{(k)}=0$?
\end{question}

\bibliographystyle{plain} 
\bibliography{WeakNil}

\end{document}